\documentclass{amsart}
\usepackage{amsmath,amssymb,amsthm, tikz}
\usepackage{xcolor}

\newtheorem{theorem}{Theorem}

\newtheorem{conjecture}[theorem]{Conjecture}

\newtheorem{question}[theorem]{Question}
\theoremstyle{definition}

\newtheorem{example}[theorem]{Example}

\newtheorem{remark}[theorem]{Remark}

\newcommand{\co}{\colon\thinspace}
\newcommand{\eps}{\varepsilon}

\DeclareMathOperator{\Len}{Len}

\DeclareMathOperator{\Vol}{Vol}
\DeclareMathOperator{\diam}{diam}
\DeclareMathOperator{\UW}{UW}
\DeclareMathOperator{\HW}{HW}

\DeclareMathOperator{\Sc}{Sc}
\DeclareMathOperator{\rk}{rk}
\DeclareMathOperator{\supp}{supp}

\begin{document}

\title[Macroscopic scalar curvature and codimension 2 width]{Macroscopic scalar curvature and \\codimension 2 width}
\author{Hannah Alpert}
\address{Auburn University, 221 Parker Hall, Auburn, AL 36849}
\email{hcalpert@auburn.edu}
\author{Alexey Balitskiy}
\address{Institute for Advanced Study, 1 Einstein Drive, Princeton, NJ 08540, and Institute for Information Transmission Problems RAS, Bolshoy Karetny per. 19, Moscow, Russia 127994}
\email{abalitskiy@ias.edu}
\author{Larry Guth}
\address{Massachusetts Institute of Technology, 77 Massachusetts Ave, Cambridge, MA 02139}
\email{lguth@math.mit.edu}

\begin{abstract}
We show that a complete $3$-dimensional Riemannian manifold $M$ with finitely generated first homology has macroscopic dimension $1$ if it satisfies the following ``macroscopic curvature'' assumptions: every ball of radius $10$ in $M$ has volume at most $4$, and every loop in every ball of radius~$1$ in $M$ is null-homologous in the concentric ball of radius~$2$.
\end{abstract}

\maketitle

\section{Introduction}

The $k$-dimensional \emph{Urysohn width} of a metric space $X$, denoted $\UW_{k}(X)$, is the infimal $\varepsilon$ such that there exists a continuous map $f \co X \rightarrow Y$, where $Y$ is a $k$-dimensional simplicial complex, for which each fiber $f^{-1}(y)$ has diameter at most $\varepsilon$. Informally, $X$ looks $(\le k)$-dimensional if one zooms out enough and disregards details at $\eps$-scale. In the case when $\diam X = \infty$, one defines the \emph{macroscopic dimension} of $X$ as the smallest nonnegative integer $k$ such that $\UW_k(X) < \infty$. The notions of width and macroscopic dimension appear in several works of Gromov in connection with various measures of ``largeness'' of a Riemannian manifold (such as hypersphericity, uniform geometric contractibility, small curvature, etc.). In particular, he formulated the following conjecture (see \cite[2.A]{gromov86}, \cite[G$_4$]{gromov88}, \cite[$2\frac12$]{gromov96}).

\begin{conjecture}[Gromov's scalar curvature conjecture]
Let $M$ be a complete Riemannian manifold of dimension $n$. If $M$ admits positive scalar curvature then $M$ has macroscopic dimension at most $n-2$. Moreover, if $\Sc(x) \ge \sigma^2 > 0$ for all $x \in M$, then $UW_{n-2}(M) < c_n/\sigma$, for some dimensional constant $c_n$.
\end{conjecture}

The conjecture is only established in dimension $3$ (see \cite{chodosh20, liokumovich20}), whereas in higher dimensions there are some partial results (see the discussion in the introduction of \cite{daher20}). It is not even known if the codimension $1$ width could be bounded from above.

Recall that the scalar curvature $\Sc(\cdot)$---the trace of the Ricci curvature tensor---is responsible for the volumetric behaviour of ``infinitesimally small'' balls:

\[
\Vol(B_r(x)) = \Vol(B_r^{\text{euc}}) - \Sc(x) \frac{r^{n+2}}{6(n+2)} + O(r^{n+4}).
\]
Here $B_r(x)$ denotes the ball of radius $r$ centered at $x$, and $\Vol(B_r^{\text{euc}})$ is the volume of the Euclidean ball of radius $r$. Therefore, the bigger the scalar curvature is, the smaller the volumes of small balls are. Taking this perspective, one can define the \emph{macroscopic scalar curvature} by measuring how small are the volumes of the balls of a certain fixed radius (as opposed to infinitesimal radius); see \cite[Section~3]{guth10} for the discussion of this metaphor. The following theorem from~\cite{guth17} can be viewed as a macroscopic cousin of a weaker form of Gromov's conjecture bounding the codimension $1$ width. 

\begin{theorem}[Macroscopic curvature and codimension $1$ width]\label{guth}
There exists a constant $\delta = \delta(n) > 0$ such that the following holds. Let $M$ be a closed Riemannian manifold of dimension $n$. If every ball of radius $1$ in $M$ has volume at most $\delta$, then $\UW_{n-1}(M) \leq 1$.
\end{theorem}

The assumption of this theorem is not sufficient to bound the codimension $2$ width. Indeed, consider the Riemannian metric on $S^3$ obtained by taking a huge disk $D^2(R)$ in $\mathbb{R}^4$, and taking the boundary of a small neighborhood to obtain $S_\varepsilon(D^2(R))$.  This product has $1$-width of order $R$, but the unit balls have volume of order $\eps$; for example, the typical unit ball looks like $D^2(1) \times S^1(\varepsilon)$.  One approach to macroscopic scalar curvature, designed to rule out examples somewhat like this one, is to assume bounds on the volume of unit balls \emph{in the universal cover} rather than the manifold itself (approach taken in~\cite{guth10, alpert17, alpert17+, sabourau20}).  However, this example is already simply connected, even though its unit balls are not.  In order to exclude this example, we would have to consider the universal covers of balls of unit or slightly larger radius, and then require bounds on the volumes of unit balls taken inside these universal covers of balls.

It is an open question whether this approach involving universal covers of balls can work, which we state as Conjecture~\ref{conj-balls} in the open questions section.  For our result in this paper, our simpler approach is to exclude the example $S_\varepsilon(D^2(R))$ and bound the codimension $2$ width by imposing an additional constraint on the metric: every short loop bounds in its neighborhood of controlled size. This is a form of the acyclicity control assumption: the homomorphism $H_1(B_1(x)) \to H_1(B_2(x))$ is trivial for all $x \in M$. The two conditions---small volume of balls plus the acyclicity control---constitute our interpretation of ``positive macroscopic scalar curvature''. 

\begin{conjecture}[Macroscopic curvature and codimension $2$ width]\label{conj-main}
There exist positive dimensional constants $\delta_n$, $C_n$, and $R_n$ such that the following holds. Let $M$ be a closed Riemannian manifold of dimension $n$.  Suppose that every ball of radius $R_n$ in $M$ has volume at most $\delta_n$, and that every loop in every ball of radius~$1$ in $M$ is (integrally) null-homologous in the concentric ball of radius~$2$.  Then $\UW_{n-2}(M) \leq C_n$.
\end{conjecture}

The two-dimensional case of this conjecture is a simple exercise whose proof is similar to the first part of the proof of~\cite[Proposition~6]{alpert17+}. We are able to prove the following weaker version of the three-dimensional case of this conjecture, with $C$ depending on the first Betti number of the manifold. 


\begin{theorem}\label{thm-main}
Let $M$ be a complete orientable $3$-dimensional Riemannian manifold, and let $\beta = \dim H_1(M; \mathbb{Q})$ be its first Betti number. Suppose that every ball of radius $10$ in $M$ has volume at most $4$, and that every loop in every ball of radius~$1$ in $M$ is (integrally) null-homologous in the concentric ball of radius~$2$.  Then $\UW_1(M) \leq 24(\beta+1)$.
\end{theorem}

Conjecture~\ref{conj-main} is a macroscopic cousin of Gromov's conjecture, but none of the two implies the other. For one direction, one can take any space to which Conjecture~\ref{conj-main} applies and perturb its metric in a small vicinity of a single point to make the scalar curvature negative, but without destroying the assumption on macroscopic curvature. For the other direction, take any closed non-simply-connected manifold with positive scalar curvature, and scale it down so that one of the nontrivial loops fits in a ball of radius $1$; then Theorem~\ref{thm-main} does not apply.

Another way to estimate the codimension $2$ width is given by a theorem of Liokumovich, Lishak, Nabutovsky, and Rotman~\cite{liokumovich21}; they proved that such a bound holds if the $(n-1)$-dimensional Hausdorff content of all unit balls is small (for a short proof, see~\cite{papasoglu20}). We remark that their theorem does not imply Theorem~\ref{thm-main}, and exhibit a manifold to which their theorem does not apply while our theorem does.

\begin{example}\label{ex-content} 
To construct a $3$-manifold satisfying the hypothesis of our main theorem, in which typical unit balls have the largest possible Hausdorff $2$-content (that is, $1$), fix parameters $\eps \ll \delta \ll 1 \ll D$. Consider a prototype ``tube'' $E\times[0,\delta]$, where $E$ is the ellipsoid $\left\{(x,y,z) \in \mathbb{R}^3 ~\middle\vert~ \frac{x^2}{\eps^2} + \frac{y^2}{\eps^2} + \frac{z^2}{\delta^2} = 1 \right\}$ with the Riemannian metric inherited from the Euclidean one. We start from a model graph $G$, and make a manifold out of tubes in the following way: to each edge of $G$ we assign a copy of the prototype tube, and at each vertex we glue the tubes along their corresponding boundaries. The gluing can be organized as follows: the edges adjacent to a given vertex are arranged cyclically, and the boundary part $\{(x,y,z) \in E ~\vert~ z \ge 0\}$ of a tube is glued to the boundary part $\{(x,y,z) \in E ~\vert~ z \le 0\}$ of the next tube by changing the sign of $z$. If there was just one adjacent edge, the boundary of the corresponding tube is stitched up by identifying $(x,y,z) \in E$ with $(x,y,-z)$. As for the model graph, consider the infinite degree $3$ tree, and take the induced subgraph on the vertices whose edge-distance from a fixed vertex is at most $D/\delta$. This way we obtain a closed $3$-manifold $M$ with the following properties: its diameter is of order $D$; its Urysohn $1$-width is of order $\delta$ (as witnessed by the evident map to the model graph $G$); the volume of any unit ball is at most of order $2^{1/\delta} \delta^2 \eps$ (this can be made arbitrarily small by picking $\eps$ depending on $\delta$); every loop in a unit ball can be contracted to a point while staying in its size $\approx \delta$ neighborhood; finally, the $2$-content of some unit balls is exactly $1$. Here is an informal estimate of the $2$-content of the typical unit ball, whose image in $G$ does not meet the vertices of $G$ of degree $1$. One can cover it by a single ball of radius $1$. Alternatively, one can cover it by balls of radius $\approx \delta$, then $\approx 2^{1/\delta}$ balls are needed, and the bound for content is $2^{1/\delta} \delta^2 \gg 1$. With even smaller balls one will still need to cover an embedded locally-almost-Euclidean surface of area $\approx 2^{1/\delta} \delta^2$. Roughly speaking, our manifold, fibered over the infinite tree, behaves ``akin to the hyperbolic plane'' at scales $\gtrsim \delta$, and the bigger balls we take the better bound for the $2$-content we get.
\end{example}

\emph{Acknowledgments.}  This material is based in part upon work supported by the National Science Foundation under Grant No. DMS-1926686 and Grant No.~DMS-1802914. Larry Guth is supported by Simons Investigator Award.
We thank Aleksandr Berdnikov, Yevgeny Liokumovich, and St\'ephane Sabourau for motivating discussions.

\section{Proof of the main result}\label{sec-3dim}

We use the following notation: $d(\cdot,\cdot)$ is the distance function; $B_r(A)$ stands for the open $r$-neighborhood of a set $A$; $S_r(A)$ denotes the boundary of $B_r(A)$, or the set of points at distance $r$ from $A$.

Theorem~\ref{thm-main} follows from the following result with $L=2$, $\varepsilon = 1$.

\begin{theorem}\label{thm-3dim}
For any $L > \varepsilon > 0$, there exists $\delta = \varepsilon L^2$ such that the following holds.  Let $M$ be a complete orientable $3$-dimensional manifold, such that every ball of radius $5L$ has volume less than $\delta$, and such that every loop contained in a ball of radius $\varepsilon$ is integrally null-homologous in the concentric ball of radius $L$.  Let $\beta = \rk H_1(M; \mathbb{Z})= \dim H_1(M; \mathbb{Q})$.  Then for any point $p \in M$, each connected component of each level set of $d(p, \cdot)$ has diameter at most $12L(\beta + 1)$.
\end{theorem}

\begin{proof}
We closely follow the proof in~\cite[Section~2]{chodosh21}, which is based on a preliminary version of the present proof; another ingredient is based on~\cite[Theorem~1.3]{balitskiy21}. Fix $p \in M$ and consider the metric spheres $S_t(p)$, $t > 0$. Suppose to the contrary that one of them, for some radius $t > 6L$, has a connected component $C$ of diameter greater than $12L(\beta + 1)$. We pretend that this component is path-connected, so that there is a path $\gamma$ in $S_t(p)$ connecting two points that are greater than distance $12L(\beta + 1)$ apart. If $C$ is not path-connected, one can pick $\gamma$ inside its small open neighborhood, and all the inequalities below will still hold up to an error summand, which can be made arbitrarily small.

We can find points $x_0, x_1, \ldots, x_{\beta + 1}$ along $\gamma$ such that the distance in $M$ between every two of them is greater than $12L$.  For each $i = 0, \ldots, \beta + 1$, let $\eta_i$ be the minimizing geodesic of length $t$, connecting $x_i$ to $p$.  
We have $\beta + 1$ triangles, each consisting of consecutive geodesics $\eta_{i-1}$ and $\eta_i$ together with the portion of $\gamma$ from $x_{i-1}$ to $x_i$.  Their classes in $H_1(M; \mathbb{Q})$ must be linearly dependent, so some nontrivial $\mathbb{Z}$-linear combination of them must be zero in homology.  We let $c$ be such a homologically trivial cycle.  There is some $i^*$ such that the net contribution of $\eta_{i^*}$ to $c$ has nonzero coefficient.  The idea of our proof is to construct a $2$-cycle that has nonzero algebraic intersection number with $c$, by intersecting $c$ transversely once along $\eta_{i^*}$ and nowhere else.

First we perturb the functions $d(x_{i^*}, \cdot)$ and $d(\eta_{i^*}, \cdot)$, approximating them by smooth functions.  Because the perturbation can be arbitrarily small, we do not keep track of it in the proof, pretending instead that these functions are already smooth.  We claim that there is a choice $(\ell_1, \ell_2) \in (0, L)\times (0, L/2)$ for which the following properties hold:
\begin{enumerate}
\item $S_{4L + \ell_1}(x_{i^*})$ is an embedded surface;
\item $S_{4L + \ell_1}(x_{i^*}) \cap S_{L + \ell_2}(\eta_{i^*})$ is an embedded union of disjoint loops; and
\item these loops have total length at most $2\varepsilon$.
\end{enumerate}
For the first two properties, by Sard's theorem we know that the regular values of the (assumed smooth) functions $d(x_{i^*}, \cdot)$ and $d(x_{i^*}, \cdot) \times d(\eta_{i^*}, \cdot)$ have full measure.  For the third property, we use the coarea inequality on the function $f(q) = (d(x_{i^*}, q) - 4L, d(\eta_{i^*}, q) -L)$, which is $1$-Lipschitz in each coordinate:
\begin{align*} \int\limits_{(r_1, r_2) \in (0, L)\times(0,L/2)} \Len(f^{-1}(r_1, r_2))\ dr_1dr_2 \leq& \Vol (f^{-1}((0, L)\times(0,L/2))\\
 \leq& \Vol (B_{5L}(x_{i^*})) < \delta.
 \end{align*}
 We may choose a pair $(\ell_1, \ell_2)$ such that the corresponding union of loops $f^{-1}(\ell_1, \ell_2)$ has at most average total length, and thus has total length less than $\frac{2\delta}{L^2} = 2\varepsilon$, while simultaneously guaranteeing that $\ell_1$ and $(\ell_1, \ell_2)$ are regular values.  Thus, the three desired properties hold.
 
 Let $\Sigma$ be the connected component of $S_{4L + \ell_1}(x_{i^*}) \cap B_{L + \ell_2}(\eta_{i^*})$ that intersects $\eta_{i^*}$.  That intersection is orthogonal and is at a single point.  Let $\Gamma$ be the union of those components of $S_{4L + \ell_1}(x_{i^*}) \cap S_{L + \ell_2}(\eta_{i^*})$ that constitute the boundary of $\Sigma$.  Because $\Gamma$ has length at most $2\varepsilon$, each of its components is contained in a ball of radius $\varepsilon$, and so our hypothesis guarantees that $\Gamma$ can be filled in its $L$-neighborhood by a $2$-cycle $\Sigma'$.  We will be done if we show that $\Gamma$ stays at least distance $L$ from $c$, because then gluing $\Sigma$ and $\Sigma'$ along $\Gamma$ produces an integral $2$-cycle with nonzero intersection number with $c$.  This would contradict our construction of $c$ as null-homologous.
 
 First we check that $\Gamma$ stays at least distance $L$ away from $\gamma$.  Let $s$ be a point on $\Gamma$.  There is a point $e$ on $\eta_{i^*}$ with $d(e, s) \leq L + \ell_2$.  By the triangle inequality on $x_{i^*}$, $s$, and $e$, we have $d(x_{i^*}, e) \geq 3L + \ell_1 - \ell_2$.  Because $d(x_{i^*}, e) + d(e, p) = t$, the triangle inequality on $e$, $p$, and $s$ gives $d(p, s) \leq t - 2L + 2\ell_2 - \ell_1\leq t - L$, and so $s$ must be at least distance $L$ from $S_t(p)$ and thus from $\gamma$.
 
 Next we check that $\Gamma$ stays at least distance $L$ away from each $\eta_j$ with $j \neq i^*$.  Let $s$ be a point on $\Gamma$, and suppose that there is a point $e$ on $\eta_j$ with $d(e, s) \leq L$.  Then the triangle inequality on $x_{i^*}$, $s$, and $e$ gives $d(x_{i^*}, e) \leq 5L + \ell_1$, so the triangle inequality on $x_{i^*}$, $e$, and $p$ gives $d(p, e) \geq t - (5L + \ell_1)$.  Because $d(p, e) + d(e, x_j) = t$, we have $d(e, x_j) \leq 5L + \ell_1$.  Thus, the triangle inequality on $x_{i^*}$, $s$, and $x_j$ gives $d(x_{i^*}, x_j) \leq 10L + 2\ell_1$, contradicting the assumption $d(x_{i^*}, x_j) > 12L$.
 
 Thus, filling $\Gamma$ in its $L$-neighborhood and adding this filling to $\Sigma$ gives a $2$-cycle with nontrivial algebraic intersection with the cycle $c$, contradicting the fact that $c$ is null-homologous.
\end{proof}

\begin{remark}\label{non-orient}
If $M$ is not orientable, we can do the same proof with $\mathbb{Z}/2$ coefficients.  Instead of assuming that every loop contained in a ball of radius $\varepsilon$ is integrally null-homologous in the concentric ball of radius $L$, we assume that it is null-homologous in the sense of $\mathbb{Z}/2$ coefficients.  And, instead of taking $\beta = \dim H_1(M; \mathbb{Q})$, we take $\beta' = \dim H_1(M; \mathbb{Z}/2)$.  Then each connected component of each level set of $d(p, \cdot)$ has diameter at most $12L(\beta' + 1)$.
\end{remark}

\begin{remark}\label{cll}
The same trick can be used to prove a genus-dependent version of~\cite[Proposition~7]{chodosh21}. Namely, the following holds. Let $M$ be an complete oriented $n$-dimensional Riemannian manifold such that any integral $(n-2)$-cycle can be filled in its $L$-neighborhood. Then $\UW_{n-2}(M) \le 12L(\beta + 1)$, where $\beta = \dim H_1(M; \mathbb{Q})$ is the first Betti number. (In~\cite{chodosh21}, $M$ is the universal cover of a manifold, so $\beta=0$.) Similarly to Remark~\ref{non-orient}, a mod~$2$ version of this result also holds.
\end{remark}

\section{Open questions}

The main question we are interested in is whether Conjecture~\ref{conj-main} holds. To start off, we would be interested to know if the dependence on the Betti number is redundant in Theorem~\ref{thm-main}.  We would also be interested to resolve an alternative version of Conjecture~\ref{conj-main}, stated in terms of universal covers of balls, as follows.

\begin{conjecture}\label{conj-balls}
There exist positive dimensional constants $\delta_n$, $C_n$, and $R_n$ such that the following holds. Let $M$ be a closed Riemannian manifold of dimension $n$.  For each point $p$ in $M$, we consider the universal cover of the ball $B_{2R_n}(p)$, and we consider the volume of the ball $B_{R_n}(\widetilde{p})$ in this universal cover, where $\widetilde{p}$ denotes a lift of $p$.  Suppose that this ball has volume at most $\delta_n$ for every choice of $p$.  Then $\UW_{n-2}(M) \leq C_n$.
\end{conjecture}

Our proof approach in this paper is not enough to prove a result with a hypothesis of this form, because the proof relies on two different radii: $5L$, which is the radius of balls that are guaranteed to have small volume, and a smaller $L$, which is the radius of a filling of a short loop.  In the alternative conjecture, the constant $R_n$ takes the role of $5L$, and $2R_n$ takes the role of $L$, so the filling may extend too far for our proof to work.

To remove the dependence on Betti number in Theorem~\ref{thm-main}, one approach could be to apply the theorem to the universal cover of our $3$-manifold, and then hope that if the universal cover has small $1$-width, then so does the original manifold.  Deciding whether this approach is viable would require resolving the following question.

\begin{question}
Do there exist $3$-dimensional closed Riemannian manifolds $M$ such that $\UW_1(\widetilde{M})$ is bounded, but $\UW_1(M)$ is arbitrarily large?
\end{question}


Even if the answer is negative, a similar approach to Conjecture~\ref{conj-main} will have to fail in higher dimensions, as there are closed Riemannian manifolds $M^n$, $n \ge 4$, such that $\UW_{n-2}(M^n) \gg \UW_{n-2}(\widetilde M^n)$. We exhibit such a manifold for $n=4$. In higher dimensions, one can take the product of this example with, say, a large round sphere.
\begin{example}
In $\mathbb{R}^5$, take the standard cubic grid with the vertices in $\mathbb{Z}^5$. Let $Z$ be the two-dimensional skeleton of this grid, and let $Z^\vee$ be the two-dimensional skeleton of the dual grid, that is, $Z^\vee = Z + (\frac12, \ldots, \frac12)$. Let $\widetilde{M}$ consist of the points
equidistant from $Z$ and $Z^\vee$. After a slight smoothing, $\widetilde{M}$ becomes a $4$-dimensional Riemannian manifold. Notice that it is simply-connected, since it is homotopy equivalent to $N = \mathbb{R}^5 \setminus (Z \cup Z^\vee)$, and every loop in $N$ can be contracted avoiding $Z \cup Z^\vee$. Pick an integer $R \gg 1$. The manifold $M$ is defined as the quotient of $\widetilde{M}$ by the lattice $\Lambda$ generated by the following vectors: $v_1 = (R,0,0,0,0)$, $v_2 = (0,R,0,0,0)$, $v_3 = (0,0,R,0,0)$, $v_4 = (0,0,0,R,0)$, $v_5 = (\frac12,\frac12,\frac12,\frac12,\frac12+R)$. These five  translations preserve $\widetilde{M}$, and the last one swaps $Z$ and $Z^\vee$. Gromov's ``fiber contraction'' argument can be used to show that $\UW_3(M)$ is of order $R$. Namely, one can apply \cite[Corollary~2.3]{balitskiy21} to the evident inclusion of $M$ in the $5$-dimensional torus
$\mathbb{R}^5 / \Lambda$; it is non-trivial at the level of $H_4(\cdot;\mathbb{Z}/2)$, since a generic circle parallel to $v_5$ in the torus intersects $M$ an odd number of times. Therefore, the width $\UW_2(M)$, sandwiched between $\UW_3(M)$ and $\diam M$, is of order $R$ as well. As for the universal cover $\widetilde{M}$, it can be projected to $Z$ with fibers of size $\approx 1$, so $\UW_2(\widetilde M)$ is of order $1$.
\end{example}




The acyclicity control assumption is somewhat reminiscent of the conclusion of~\cite[Theorem~10.7]{gromov83}, where it is proven that in presence of positive scalar curvature \emph{every null-homologous loop bounds nearby}. This type of assumption for (not just $1$-dimensional) cycles, without any volumetric bounds, might be a good candidate for the role of ``positive macroscopic curvature''. Inspired by the work of Alexandrov~\cite{alexandroff33}, one can define the \emph{$k$-dimensional homological width} of a path metric space $X$ as follows: 
\[
\HW_k(X; R) = \sup\limits_{Z \in \mathcal{B}_k(X;R)} \inf\{r \ge 0 ~\vert~ Z \text{ bounds a $(k+1)$-chain in } B_r(\supp Z)\}.
\]
Here $R$ is the coefficient ring of singular homology; $\mathcal{B}_k(X;R)$ is the space of $k$-boundaries (null-homologous singular $k$-cycles) in $X$; $\supp Z$ is the subset of $X$ where the images of the singular simplices of $Z$ land. 

\begin{conjecture}[cf.~\cite{alexandroff33}]
For some dimensional constant $c_k$,
\[\UW_k(X) \le c_k \HW_k(X; \mathbb{Q}/\mathbb{Z}).\]
\end{conjecture}

A version of this question for $k=1$ was asked in~\cite{gromov20}, apparently assuming integer coefficients. We notice though that $\UW_1(X)$ can be much greater than $\HW_1(X; \mathbb{Z})$. Such an example was given by Pontryagin \cite{pontryagin30} in the context of dimension theory: the \emph{Pontryagin surfaces} have topological dimension $2$ but homological dimension $1$ (modulo various coefficients). The Pontryagin surfaces are not manifolds, but a finite approximation of one of them is a manifold and will do for our purposes.\footnote{Aleksandr Berdnikov came up with this example without knowing about Pontryagin's work.}

\begin{example}
Take a large sphere $S^2(R)$ with round metric and distribute points $p_i \in S^2(R)$ so that the balls $B_1(p_i)$ do not overlap but the balls $B_3(p_i)$ cover the sphere. At each $p_i$, punch a small hole and glue a cross-cap (a M\"obius band) instead. We use the same fixed cross-cap model of diameter less than $1$. The resulting connected sum of projective planes is a Riemannian manifold $M$ satisfying the following properties. First, $\UW_1(M)$ is of order $R$, as one can show using ``fiber contraction''; \cite[Corollary~2.3]{balitskiy21} can be applied to the pinching map $M \to S^2(R)$ and mod $2$ homology. Second, $\HW_1(M; \mathbb{Z})$ is of order $1$, as we now show. Suppose $\Sigma$ is an integral $2$-chain whose boundary $1$-cycle is $\gamma$. Suppose there is a point $p \in \Sigma$ that is distance at least $4$ away from $\gamma$. There is an open neighborhood $p \in U \subset B_4(p)$ that contains exactly one of the cross-caps. The image of $\Sigma$ in the relative homology $H_2(M, M\setminus U; \mathbb{Z})$ must be non-zero, but this homology is trivial, so we get a contradiction.
\end{example}

\bibliography{codim2-bib}
\bibliographystyle{amsalpha}
\end{document}